\documentclass[leqno]{amsart}
\usepackage{amsmath}
\usepackage{amssymb}
\usepackage{amsthm}
\usepackage{enumerate}
\usepackage[mathscr]{eucal}
\theoremstyle{plain}
\usepackage{tikz}
\newtheorem{theorem}{Theorem}[section]

\theoremstyle{definition}

\newtheorem{remark}[theorem]{Remark}

\newtheorem{cor}[theorem]{Corollary}
\theoremstyle{remark}

\begin{document}

\title [On numerical radius and Crawford  number attainment sets]{On numerical radius and Crawford  number attainment sets of a bounded linear operator}

\author{Debmalya Sain, Arpita Mal, Pintu Bhunia and Kallol Paul }

\address{(Sain) Department of Mathematics, Indian Institute of Science, Bengaluru 560012,
Karnataka, India.}
\email{saindebmalya@gmail.com}

\address{(Mal) Department of Mathematics\\ Jadavpur University\\ Kolkata 700032\\ West Bengal\\ India.}
\email{arpitamalju@gmail.com}

\address{(Bhunia) Department of Mathematics, Jadavpur University, Kolkata 700032, India.}
\email{pintubhunia5206@gmail.com}

\address{(Paul) Department of Mathematics, Jadavpur University, Kolkata 700032, India.}
\email{kalloldada@gmail.com}

\thanks{The research of Dr. Debmalya Sain is sponsored by Dr. D. S. Kothari Post-doctoral Fellowship, under the mentorship of Professor Gadadhar Misra.  Dr. Sain feels elated to acknowledge the loving guidance and inspirations of Swami Sivapradananda, his high school Headmaster. Miss Arpita Mal and  Mr. Pintu Bhunia would like to thank UGC, Govt. of India for the financial support in the form of SRF and JRF respectively. Prof. Kallol Paul would like to thank RUSA 2.0, Jadavpur University for the financial support.}




\date{}
\maketitle
\begin{abstract}
We completely characterize the Crawford number attainment set and the numerical radius attainment set of a bounded linear operator on a Hilbert space. We study the intersection properties of the corresponding attainment sets of numerical radius, Crawford number, norm, minimum norm of a bounded linear operator defined on a normed space. Our study illustrates the similarities and the differences of the extremal properties of a bounded linear operator on a Hilbert space and a general normed space.
\end{abstract}

\section{Introduction}
The present article is devoted to studying some distinguished subsets of the unit sphere of a normed space, and their intersection properties, in connection with a given bounded linear operator defined on the space. Let us first mention the notations and the terminologies used in the article. \\

The letters $\mathbb{X}, \mathbb{H}$ stand for a normed space and a Hilbert space respectively. In this article, we consider the ground field to be either the set of real numbers $ \mathbb{R}, $ or, the set of complex numbers $ \mathbb{C}. $ Given $ z \in \mathbb{C}, $ we follow the usual convention to denote the real part of $ z $ and the imaginary part of $ z $ by Re $ z $ and Im $ z $ respectively. The complex conjugate of $ z $ is denoted by $ \overline{z}. $ Let $ B_{\mathbb{X}}= \{ x \in \mathbb{X}~:\|x\| \leq 1 \} $ and $  S_{\mathbb{X}}= \{ x \in \mathbb{X}~:\|x\| = 1 \} $ denote the unit ball and the unit sphere of $ \mathbb{X} $ respectively. Let $L(\mathbb{X})$ denote the collection of all bounded linear operators from $\mathbb{X}$ to $ \mathbb{X} $ and let $\mathbb X^*$ denote the dual space of $\mathbb{X}.$ Given a bounded linear operator $ T \in L(\mathbb{X}), $ there are certain subsets of $ S_{\mathbb{X}} $ and certain numerical constants associated with $ T, $ which are important in understanding the way $ T $ acts on $ \mathbb{X}. $ In this article, we explore some of these concepts and study the intersection properties of some sets associated with these concepts. Given $ T \in L(\mathbb{X}), $ the numerical range $ V(T) $, the numerical radius $ v(T), $ the Crawford number $c(T),$ the norm  $\|T\|$ and the minimum norm $m(T)$  of $T$ are defined  respectively as
\begin{eqnarray*}
	V(T)&= &\{ x^*(Tx) : x\in S_\mathbb X,x^* \in S_{\mathbb X^*},x^*(x)=1\},\\
	v(T)&=&\sup \{ | x^*(Tx) |~:x\in S_\mathbb X,x^* \in S_{\mathbb X^*},x^*(x)=1\},\\
	c(T)&=&\inf \{ | x^*(Tx) |~:x\in S_\mathbb X,x^* \in S_{\mathbb X^*},x^*(x)=1\},\\
	\|T\|&=&\sup \{\|Tx\|~:x \in S_{\mathbb{X}}\},\\
	m(T) &=&\inf \{\|Tx\|~:x \in S_{\mathbb{X}}\}.
\end{eqnarray*}
 We refer the readers to \cite{GR,SPM} for more information on these concepts and their importance in studying a bounded linear operator between Hilbert spaces and normed spaces. We would like to observe that other than $ V(T), $ which is a subset of the ground field, each of the above concepts is essentially a non-negative constant associated with the bounded linear operator $ T $ acting on $ \mathbb{X}. $ Therefore, a natural question in this context is to seek for a complete characterization of the points on the unit sphere of $ \mathbb{X} $ at which $ T $ attains these constants. This motivates the introduction of the numerical radius attainment set $V_T,$ the Crawford number attainment set $c_T,$ the norm attainment set $M_T$ and the minimum norm attainment set $m_T,$ corresponding to a bounded linear operator $T \in L(\mathbb{X})$ in the following natural way:
\begin{eqnarray*}
	V_T&=&\{ x \in S_{\mathbb{X}}~:~ \exists~ x^{*}\in S_{\mathbb{X^{*}}} ~\mbox{such that}~ |x^{*}(Tx)|=v(T) ~\mbox{and}~ x^{*}(x)=1 \},\\
	c_T&=&\{ x \in S_{\mathbb{X}}~:~ \exists~ x^{*}\in S_{\mathbb{X^{*}}} ~\mbox{such that}~ |x^{*}(Tx)|=c(T) ~\mbox{and}~ x^{*}(x)=1 \},\\
	M_T&=&\{ x \in S_{\mathbb{X}}~:~\|Tx\|=\|T\|\},\\
	m_T&=&\{ x \in S_{\mathbb{X}}~:~\|Tx\|=m(T)\}.
\end{eqnarray*}

We would like to mention here that a complete characterization of $ M_T $ has been obtained in \cite[Th. 2.1]{S} for operators on a Hilbert space. On the other hand, a complete characterization of both $ M_T $ and $ m_T $ has been obtained in \cite{SPM} for operators between Banach spaces. However, there is no known characterization of $ V_T $ or $ c_T, $ for a given bounded linear operator $ T. $ In the present article, we obtain an answer to this problem for a bounded linear operator on a Hilbert space, separately in the real case and the complex case. Let us recall in this context that given $ T \in L(\mathbb{H}), $ Re$ (T) $  and Im$ (T) $ denotes the real part and imaginary part of $ T, $ i.e., $ \textit{Re}(T) = \frac{T+T^*}{2},  \textit{Im}(T) = \frac{T-T^*}{2i}, $ where $ T^* $ is the adjoint operator of $ T $ and $ i^2 = -1. $ Using these notions, we obtain a complete characterization of the Crawford number attainment set of a bounded linear operator on a real (complex) Hilbert space. In fact, in case of complex Hilbert spaces, we obtain two equivalent characterizations of the Crawford number attainment set of a bounded linear operator. In similar spirit, we continue the study of the numerical radius attainment set of a bounded linear operator on a Hilbert space, and obtain complete characterizations of the same, separately in the real case and the complex case. As an immediate consequence of our exploration, we discuss the intersection properties of some special subsets of the unit sphere of a normed space, especially in the context of bounded linear operators on a Hilbert space. Whenever possible, we also discuss the aforementioned problems in the setting of normed spaces. Our study illustrates the difference between the geometry of Hilbert spaces and the geometry of normed spaces, from the perspective of investigating optimization problems for a bounded linear operator. A normed space $ \mathbb{X} $ is strictly convex if every unit vector in $ \mathbb{X} $ is an extreme point of the unit ball $ B_{\mathbb{X}}. $ A bounded linear operator $ T \in L(\mathbb{X}) $ is said to satisfy the Daugavet equation if $ \| I + T \| = 1 + \| T \|, $ where $ I $ is the identity operator on $ \mathbb{X}. $\\
We also recall that a bounded linear operator $ T \in L(\mathbb{H}) $ is said to be a partial isometry if $ T $ is an isometry on the orthogonal complement of its kernel. Given a partial isometry on a Hilbert space, the orthogonal complement of its kernel is called the initial subspace and the range of the operator is called the final subspace. 

\section{Characterization of Crawford number and Numerical radius attainment vector}

We begin with a complete characterization of the Crawford number attainment set of vectors of a bounded linear operator on a real Hilbert space.

\begin{theorem}\label{Th-ct}
	Let $\mathbb{H}$ be a real Hilbert space and $T\in L(\mathbb{H}).$ Let $x\in S_{\mathbb{H}}.$ Then $x\in c_T$  if and only if either of the following conditions holds:\\
	(i) $\langle Tx,x\rangle=0,$\\
	(ii) if $\langle Tx,x\rangle\neq 0,$ then $x$ is an eigenvector of $Re(T)$ and $|\langle Tx,x\rangle| \leq |\langle Tz,z\rangle|,$ for any $z\in x^{\perp}\cap S_{\mathbb{H}}.$ 
\end{theorem}	
\begin{proof}
	First we prove the sufficient part of the theorem. $(i)$ clearly implies that $x\in c_T.$ Suppose $(ii)$ holds. It is clear from the convexity of numerical range that $\langle Tx,x\rangle \langle Tz,z\rangle > 0,$ for any $z\in S_{\mathbb{H}}.$  Let $u\in S_{\mathbb{H}}.$ Then $u=ax+bz,$ for some $a,b\in \mathbb{R}$ and $z\in x^{\perp}\cap S_{\mathbb{H}}.$ Now,
	\begin{eqnarray*}
		|\langle Tu,u\rangle |&=& |\langle aTx+bTz,ax+bz\rangle|\\
		&=& |a^2\langle Tx,x\rangle+b^2 \langle Tz,z\rangle+ab(\langle Tx,z\rangle+\langle Tz,x\rangle)|\\
		&=& |a^2\langle Tx,x\rangle+b^2 \langle Tz,z\rangle+2ab\langle Re(T)x,z\rangle|\\
		&=& |a^2\langle Tx,x\rangle+b^2 \langle Tz,z\rangle|\\
		&=& |\langle Tx,x\rangle+b^2(\langle Tz,z\rangle-\langle Tx,x\rangle)|\\
		&\geq&|\langle Tx,x\rangle|,~~\text{since } \langle Tx,x\rangle \langle Tz,z\rangle > 0 \text{ and } |\langle Tx,x\rangle| \leq |\langle Tz,z\rangle|.
	\end{eqnarray*}
	Therefore, $x\in c_T.$\\
	Now, we prove the necessary part of the theorem. Let $x\in c_T.$ Suppose $(i)$ is not true. Clearly, $|\langle Tx,x\rangle| \leq |\langle Tz,z\rangle|.$ We only show that $x$ is an eigenvector of $Re(T).$ Let $Re(T)x=ax+bz,$ where $a,b\in \mathbb{R}$ and $z\in x^{\perp}\cap S_{\mathbb{H}}.$ Consider a function $\psi_z:\mathbb{R}\to \mathbb{R}$ defined as:
	\[\psi_z(t)=\frac{|\langle T(x+tz),x+tz\rangle|}{\|x+tz\|^2}.\] 
	Now, $\langle Tx,x\rangle \neq 0\Rightarrow \psi_z$ is differentiable in a neighbourhood of $0.$ Since $x\in c_T,~\psi_z$ has a minima at $t=0.$ Therefore, $\psi_z^{'}(0)=0\Rightarrow \lim_{t\to 0}\frac{\psi_z(t)-\psi_z(0)}{t}=0\Rightarrow \langle Re(T)x,z\rangle=0\Rightarrow \langle ax+bz,z\rangle=0\Rightarrow b=0.$ Thus, $x$ is an eigenvector of $Re(T).$ This completes the proof of the theorem.
\end{proof}	

In case of a complex Hilbert space, the characterization of the Crawford number attainment set of a bounded linear operator assumes the following form:

\begin{theorem} \label{com-ct}
	Let $\mathbb{H}$ be a complex Hilbert space. Let $T\in L(\mathbb{H})$ and $x\in S_{\mathbb{H}}.$ Then $x\in c_T$ if and only if  either $\langle Tx,x\rangle=0$ or $x$ is an eigenvector of $\langle Re(T)x,x\rangle Re(T)+\langle Im(T)x,x\rangle Im(T)$ corresponding to the eigenvalue $c^2(T).$
\end{theorem}
\begin{proof}
	First we prove the necessary part of the theorem. Let $x\in c_T$. If $\langle Tx,x\rangle =0,$ then we are done. Suppose $\langle Tx,x\rangle \neq 0.$  Let $y\in S_{\mathbb{H}}$ be arbitrary. Consider the function $\psi_y:\mathbb{R}\to \mathbb{R}$ defined as 
	$$\psi_y(t)=\frac{|\langle T(x+ty),x+ty\rangle|^2}{\|x+ty\|^4}.$$ Since $x\in c_T$, so $\psi_y(t)$ has minima at $t=0.$ Since $\psi_y(t)$ is differentiable in a neighbourhood of $t=0$, so $\psi_y^{'}(0)=0$. This implies that $x$ is an eigenvector of $\langle Re(T)x,x\rangle Re(T)$ $+\langle Im(T)x,x\rangle Im(T).$ Therefore there exists a scalar $\lambda$ such that $ \big(\langle Re(T)x,x\rangle Re(T)$ $+\langle Im(T)x,x\rangle Im(T) \big) x=\lambda x.$ This shows that $\lambda=|\langle Tx,x\rangle|^2=c^2(T).$ This completes the proof of the necessary part of the theorem.
	
	Next we prove the sufficient part of the theorem. Let $x\in S_{\mathbb{H}}.$ If $\langle Tx,x\rangle =0,$ then clearly $x\in c_T.$ Suppose $x$ is an eigenvector of $\langle Re(T)x,x\rangle Re(T)+\langle Im(T)x,x\rangle Im(T)$ corresponding to the eigenvalue $c^2(T)$, i.e., $$ \big(\langle Re(T)x,x\rangle Re(T)+\langle Im(T)x,x\rangle Im(T) \big) x=c^2(T) x.$$ This implies that $c(T)=|\langle Tx,x\rangle|.$ So $x\in c_T$. This completes the proof of the theorem.
\end{proof}

Now, for real Hilbert space, the following result can be easily obtained. 
\begin{cor}\label{real-ct}
	Let $\mathbb{H}$ be a real Hilbert space. Let $T\in L(\mathbb{H})$ and $x\in S_{\mathbb{H}}.$ Then $x\in c_T$ if and only if  either $\langle Tx,x\rangle=0$ or $x$ is an eigenvector of $ Re(T)$ corresponding to the eigenvalue $\lambda$ with  $| \lambda |=c(T).$
\end{cor}

Our next result gives another characterization for the Crawford number attainment set of a bounded linear operator on a complex Hilbert space.

\begin{theorem}\label{th-ctcom}
	Let $\mathbb{H}$ be a complex Hilbert space. Let $T\in L(\mathbb{H}).$ Let $c_T>0$ and $x\in S_{\mathbb{H}}.$ Then $x\in c_T$ if and only if for arbitrary $y\in x^{\perp}\cap S_{\mathbb{H}},$ either of the following conditions holds:\\
	(i) $\langle Ty,x\rangle=\overline{\langle Tx,y\rangle}$ and $||a|^2\langle Tx,x\rangle+|b|^2\langle Ty,y\rangle+2Re\{a\overline{b}\langle Tx,y\rangle\}|\geq |\langle Tx,x\rangle|,$ for all $a,b\in \mathbb{C}$ with $|a|^2+|b|^2=1.$\\
	(ii) $\langle Ty,x\rangle=-\overline{\langle Tx,y\rangle}$ and $||a|^2\langle Tx,x\rangle+|b|^2\langle Ty,y\rangle+2iIm\{a\overline{b}\langle Tx,y\rangle\}|\geq |\langle Tx,x\rangle|,$ for all $a,b\in \mathbb{C}$ with $|a|^2+|b|^2=1.$
\end{theorem}
\begin{proof}
First we prove the easier sufficient part of the theorem. Let $u\in S_{\mathbb{H}}.$ Then there exist $a,b\in \mathbb{C}$ and $y\in x^{\perp}\cap S_{\mathbb{H}},$ such that $u=ax+by.$ Suppose for this $y\in x^{\perp}\cap S_{\mathbb{H}}$ $(i)$ holds. Then
\begin{eqnarray*}
|\langle Tu,u\rangle|&=&|\langle T(ax+by),ax+by\rangle|\\
&=&||a|^2\langle Tx,x\rangle+|b|^2\langle Ty,y\rangle+a\overline{b}\langle Tx,y\rangle+\overline{a}b\langle Ty,x\rangle|\\
&=&||a|^2\langle Tx,x\rangle+|b|^2\langle Ty,y\rangle+2Re\{a\overline{b}\langle Tx,y\rangle\}|\\
&\geq& |\langle Tx,x\rangle|
\end{eqnarray*}
Similarly, if $(ii)$ holds, then $|\langle Tu,u\rangle|\geq|\langle Tx,x\rangle|$ holds. Thus, $x\in c_T.$\\
Now, we	prove the necessary part of the theorem. Let $x\in c_T.$ Let $y\in x^{\perp}\cap S_{\mathbb{H}}.$ Then $|\langle Tx,x\rangle |>0,$ since $c_T>0.$ If possible, suppose that $\langle Ty,x\rangle\neq \overline{\langle Tx,y\rangle}$ and $\langle Ty,x\rangle\neq -\overline{\langle Tx,y\rangle}.$ Now, it is possible to find $a,b\in \mathbb{C}$ such that $|a|^2+|b|^2=1$ and $arg\{a\overline{b}\langle Tx,y\rangle+\overline{a}b\langle Ty,x\rangle\}=\pi+arg\{\langle Tx,x\rangle\}.$ It is easy to check that for sufficiently small $|b|,$ $|a\overline{b}\langle Tx,y\rangle+\overline{a}b\langle Ty,x\rangle|>|b|^2|\langle Ty,y\rangle|.$ Therefore, we have,
\begin{eqnarray*}
	|\langle T(ax+by),ax+by\rangle|&=&||a|^2\langle Tx,x\rangle+|b|^2|\langle Ty,y\rangle+a\overline{b}\langle Tx,y\rangle+\overline{a}b\langle Ty,x\rangle|\\
	&\leq& ||a|^2|\langle Tx,x\rangle|-|a\overline{b}\langle Tx,y\rangle+\overline{a}b\langle Ty,x\rangle||+|b|^2|\langle Ty,y\rangle|\\
	&=& |a|^2|\langle Tx,x\rangle|-|a\overline{b}\langle Tx,y\rangle+\overline{a}b\langle Ty,x\rangle|+|b|^2|\langle Ty,y\rangle|\\
	&<& |a|^2|\langle Tx,x\rangle|\\
	&<& |\langle Tx,x\rangle|,
\end{eqnarray*}	
contradicting that $x\in c_T.$ Therefore, either $\langle Ty,x\rangle = \overline{\langle Tx,y\rangle}$ or $\langle Ty,x\rangle = -\overline{\langle Tx,y\rangle}.$ Now, let $\langle Ty,x\rangle = \overline{\langle Tx,y\rangle}.$ Then for all $a,b\in \mathbb{C}$ such that $|a|^2+|b|^2=1,$ $|\langle T(ax+by),ax+by\rangle|\geq |\langle Tx,x\rangle|\Rightarrow ||a|^2\langle Tx,x\rangle+|b|^2\langle Ty,y\rangle+2Re\{a\overline{b}\langle Tx,y\rangle\}|\geq |\langle Tx,x\rangle|.$ Similarly, for $y\in x^{\perp}\cap S_{\mathbb{H}},$ $\langle Ty,x\rangle = -\overline{\langle Tx,y\rangle}$ implies that for all $a,b\in \mathbb{C}$ such that $|a|^2+|b|^2=1,$ $||a|^2\langle Tx,x\rangle+|b|^2\langle Ty,y\rangle+2iIm\{a\overline{b}\langle Tx,y\rangle\}|\geq |\langle Tx,x\rangle|.$ This completes the proof of the theorem.
\end{proof}

We next obtain a complete characterization of the numerical radius attainment set of a bounded linear operator on a Hilbert space. First we treat the case of real Hilbert spaces.

\begin{theorem}\label{Th-vt}
	Let $\mathbb{H}$ be a real Hilbert space and $T\in L(\mathbb{H}).$ Let $x\in S_{\mathbb{H}}.$ Then $x\in V_T$  if and only if for any $z\in x^{\perp}\cap S_{\mathbb{H}},$ the following conditions hold:\\
	(i) $x$ is an eigenvector of $Re(T).$\\
	(ii) $|\langle Tx,x\rangle| \geq |\langle Tz,z\rangle|.$ 
\end{theorem}
\begin{proof}
	We first prove the necessary part of the theorem. Now, $\langle Ty,y\rangle=\langle Re(T)y,y\rangle$ for every $y\in \mathbb{H}.$ Therefore, $x\in V_T\Rightarrow x\in V_{Re(T)}.$ Let $\lambda=\langle Re(T)x,x\rangle.$ Then $|\lambda|=v(Re(T))=\|Re(T)\|,$ since $Re(T)$ is self adjoint. Now, $\|Re(T)\|=|\lambda|=|\langle Re(T)x,x\rangle|\leq \|Re(T)x\|\|x\|\leq \|Re(T)\|.$ Thus, $|\langle Re(T)x,x\rangle |= \|Re(T)x\|\|x\|\Rightarrow Re(T)x=\lambda x.$ Hence,  $x$ is an eigenvector of $Re(T).$ $(ii)$ follows from the definition of $V_T.$ This completes the proof of the necessary part of the theorem.\\
	Now, we prove the sufficient part of the theorem. Let $u\in S_{\mathbb{H}}.$ Then $u=ax+bz,$ for some $a,b\in \mathbb{R}$ and $z\in x^{\perp}\cap S_{\mathbb{H}}.$ Let $Re(T)x=\lambda x$ for some $\lambda\in \mathbb{R}.$ Then $\langle Tx,z\rangle+\langle Tz,x\rangle=\langle 2Re(T)x,z\rangle=\langle 2\lambda x,z\rangle=0.$ Now,
	\begin{eqnarray*}
		|\langle Tu,u\rangle |&=& |\langle aTx+bTz,ax+bz\rangle|\\
		&=& |a^2\langle Tx,x\rangle+b^2 \langle Tz,z\rangle|\\
		&=& |\langle Tx,x\rangle+b^2(\langle Tz,z\rangle-\langle Tx,x\rangle)|.
	\end{eqnarray*}
	We show that whatever be the sign of $\langle Tx,x\rangle,$  $|\langle Tu,u\rangle|\leq |\langle Tx,x\rangle|.$ Let $\langle Tx,x\rangle \geq 0.$ Then
	\begin{eqnarray*}
		&&|\langle Tz,z\rangle|\leq |\langle Tx,x \rangle|\\
		&\Rightarrow &-\langle Tx,x \rangle \leq \langle Tz,z\rangle\leq \langle Tx,x \rangle\\
		&\Rightarrow& -2b^2\langle Tx,x\rangle\leq b^2(\langle Tz,z\rangle-\langle Tx,x\rangle)\leq 0\\
		&\Rightarrow& \langle Tx,x \rangle-2b^2\langle Tx,x \rangle \leq  \langle Tx,x \rangle+b^2( \langle Tz,z\rangle-\langle Tx,x\rangle)\leq \langle Tx,x \rangle\\
		&\Rightarrow &-\langle Tx,x \rangle \leq \langle Tx,x \rangle+b^2( \langle Tz,z\rangle-\langle Tx,x\rangle)\leq \langle Tx,x \rangle\\
		&\Rightarrow &-\langle Tx,x \rangle \leq \langle Tu,u \rangle \leq \langle Tx,x \rangle\\
		&\Rightarrow&|\langle Tu,u\rangle|\leq |\langle Tx,x\rangle|.
	\end{eqnarray*}
	Similarly if $\langle Tx,x\rangle \leq 0,$ then $|\langle Tu,u\rangle|\leq |\langle Tx,x\rangle|.$ Therefore, $x\in V_T.$
\end{proof}

In case of a complex Hilbert space, the above characterization theorem assumes the following form:

\begin{theorem} \label{com-vt}
	Let $\mathbb{H}$ be a complex Hilbert space. Let $T\in L(\mathbb{H})$ and $x\in S_{\mathbb{H}}.$ Then $x\in V_T$ if and only if  $x$ is an eigenvector of $\langle Re(T)x,x\rangle Re(T)+\langle Im(T)x,x\rangle Im(T)$ corresponding to the eigenvalue $v^2(T).$
\end{theorem}
\begin{proof}
	The proof follows similarly as the proof of Theorem \ref{com-ct}.
\end{proof}

If the Hilbert space is real, then one can easily obtain the following result.

\begin{cor}\label{real-vt}
	Let $\mathbb{H}$ be a real Hilbert space. Let $T\in L(\mathbb{H})$ and $x\in S_{\mathbb{H}}.$ Then $x\in V_T$ if and only if  $x$ is an eigenvector of $ Re(T)$ corresponding to the eigenvalue $\lambda$ with  $| \lambda |=v(T).$
\end{cor}

Analogous to Theorem \ref{th-ctcom}, it is possible to obtain another characterization for the numerical radius attainment set of  a bounded linear operator on a complex Hilbert space.

\begin{theorem}
	Let $\mathbb{H}$ be a complex Hilbert space. Let $T\in L(\mathbb{H}).$ Let $x\in S_{\mathbb{H}}.$ Then $x\in V_T$ if and only if for arbitrary $y\in x^{\perp}\cap S_{\mathbb{H}},$ either of the following conditions holds:\\
	(i) $\langle Ty,x\rangle=\overline{\langle Tx,y\rangle}$ and $||a|^2\langle Tx,x\rangle+|b|^2\langle Ty,y\rangle+2Re\{a\overline{b}\langle Tx,y\rangle\}|\leq |\langle Tx,x\rangle|,$ for all $a,b\in \mathbb{C}$ with $|a|^2+|b|^2=1.$\\
	(ii) $\langle Ty,x\rangle=-\overline{\langle Tx,y\rangle}$ and $||a|^2\langle Tx,x\rangle+|b|^2\langle Ty,y\rangle+2iIm\{a\overline{b}\langle Tx,y\rangle\}|\leq |\langle Tx,x\rangle|,$ for all $a,b\in \mathbb{C}$ with $|a|^2+|b|^2=1.$
\end{theorem}

\section{Intersection property}
In this section, we study the intersection properties of the above mentioned subsets of the unit sphere of a normed space $ \mathbb{X}, $ in connection with a bounded linear operator $ T $ defined on $ \mathbb{X}. $ Some of the results follow directly from our study in the previous section. Therefore, we omit the proofs whenever they follow immediately as a consequence of our results in the previous section. We begin with a characterization of the intersection of the norm attainment set $ M_T, $ or the minimum norm attainment set $ m_T, $ and the numerical radius attainment set $ V_T $ of a bounded linear operator $ T $ on a Hilbert space $ \mathbb{H}. $ First we treat the case of real Hilbert spaces.

\begin{theorem}
	Let $\mathbb{H}$ be a real Hilbert space and $T\in L(\mathbb{H}).$ Then $M_T\cap V_T\neq \emptyset$ if and only if there exists $x\in S_{\mathbb{H}}$ such that for any $z\in x^{\perp}\cap S_{\mathbb{H}},$ the following conditions hold:\\
	(i) $x$ is an eigenvector of $Re(T).$\\
	(ii) $|\langle Tx,x\rangle| \geq |\langle Tz,z\rangle|,$\\
	(iii) $\langle Tx,Tz\rangle=0,$\\
	(iv) $\|Tx\|\geq \|Tz\|.$  
\end{theorem}	

\begin{proof}
	The result follows from Theorem \ref{Th-vt} and \cite[Th. 2.1]{S}. 
\end{proof}

\begin{theorem}
	Let $\mathbb{H}$ be a real Hilbert space and $T\in L(\mathbb{H}).$ Then $m_T\cap V_T\neq \emptyset$ if and only if there exists $x\in S_{\mathbb{H}}$ such that for any $z\in x^{\perp}\cap S_{\mathbb{H}},$ the following conditions hold:\\
	(i) $x$ is an eigenvector of $Re(T).$\\
	(ii) $|\langle Tx,x\rangle| \geq |\langle Tz,z\rangle|,$\\
	(iv) $\langle Tx,Tz\rangle=0,$\\
	(v) $\|Tx\|\leq \|Tz\|.$ 
\end{theorem}
	
\begin{proof}
	The result follows from Theorem \ref{Th-vt} and \cite[Th. 2.4]{SPM}. 
\end{proof}

If the underlying field is either real or complex, we have the following result:

\begin{theorem}\label{theorem-general}
	Let $\mathbb{H}$ be a Hilbert space. Let $T\in L(\mathbb{H}).$ Then $M_T\cap V_T\neq \emptyset$ if and only if the following conditions hold:\\
	$(i)$ $\|T\|=\|T|_Y\|,$\\
	$(ii)$ $v(T)=v(T|_Y),\\$
	$(iii)$ $M_{T|_Y}\cap V_{T|_{\mathbb{Y}}}\neq \emptyset,$ where $Y=(\ker(T))^\perp.$
\end{theorem}

\begin{proof}
First we prove the sufficient part of the theorem. Let $y\in M_{T|_{\mathbb{Y}}}\cap V_{T|_{\mathbb{Y}}}$. Then, $\|T|_{\mathbb{Y}}y\|=\|T|_{\mathbb{Y}}\|=\|T\|.$ This implies that $y\in M_T.$ Also, $|\langle T|_{\mathbb{Y}}y,y\rangle|=v(T|_{\mathbb{Y}})=v(T)$. This implies that $y\in V_T.$ Therefore, $M_T \cap V_T \neq \emptyset$.\\ 
	For the necessary part, suppose that $e\in M_T\cap V_T.$ Now, $H=Y\oplus Y^\perp,$ where $Y=(\ker(T))^\perp.$ Clearly, $Y^\perp=\ker(T).$ Let $e=x+y,$ where $x\in Y,y\in Y^\perp.$ Then $1=\|e\|^2=\|x\|^2+\|y\|^2.$ Thus, $\|x\|\leq 1.$ Now, $\|T\|=\|Te\|=\|Tx+Ty\|=\|Tx\|.$ Thus, $\|x\|=1$ and $y=0.$ This gives that $e=x\in Y.$ Hence, $\|T\|=\|Te\|\leq\|T|_Y\|\leq \|T\|\Rightarrow \|T\|=\|T|_Y\|.$ Thus, $(i)$ is proved. Now, $v(T)=|\langle Te,e \rangle|\leq v(T|_Y)\leq v(T)\Rightarrow v(T)=v(T|_Y).$ Thus, $(ii)$ is proved. Clearly, $e\in M_{T|_Y}\cap V_{T|Y}.$ Thus, $M_{T|_Y}\cap V_{T|Y}\neq\emptyset.$ This completes the proof of the theorem.   
\end{proof}
 
On the other hand, in case of a real Hilbert space, we also have the following characterization: 

\begin{theorem}\label{th-mtvt}
Let $\mathbb{H}$ be a real Hilbert space and $T\in L(\mathbb{H}).$ Then $M_T\cap V_T\neq\emptyset$ if and only if either of the following conditions holds:\\
(i) there exists $x\in S_{\mathbb{H}}$ such that $Tx=\pm \|T\|x.$\\
(ii) there exists a two dimensional subspace $Y$ of $(ker(T))^{\perp}$ such that $T:Y\to Y$ is a scalar multiple of isometry, $\|T\|=\|T|_{Y}\|$ and $v(T)=v(T|_Y).$	
\end{theorem}
\begin{proof}
	First we prove the necessary part of the theorem. Let $M_T\cap V_T\neq\emptyset.$ Suppose that $x\in  M_T\cap V_T.$ If $\{Tx,x\}$ is linearly dependent, then there exists $\lambda\in \mathbb{R}$ such that $Tx=\lambda x.$ Now, $x\in M_T\Rightarrow \|T\|=|\lambda|.$ Thus, in this case $(i)$ holds. If $\{Tx,x\}$ is linearly independent, then consider $Y= span\{Tx,x\}.$ $x\in M_T\cap V_T\cap Y\Rightarrow x\in M_{T|_Y}\cap V_{T|_Y}.$ Let $y\in m_{T|_Y}.$ Then by \cite[Th. 3.1, Th. 2.4]{SPM}, we have, $x\perp y$ and $Tx\perp Ty.$  Suppose that $Tx=\alpha x+\beta y,$ for some $\alpha,\beta\in \mathbb{R}.$ Then $\|T\|=\sqrt{\alpha^2+\beta^2.}$ Since  $\{Tx,x\}$ is linearly independent, $\beta \neq 0.$  Since, $Tx\perp Ty,$ there exists $z\in Y^{\perp}\cap S_{\mathbb{H}}$ such that $Ty=\mu(\beta x-\alpha y)+\gamma z,$ for some $\mu,\gamma \in \mathbb{R}.$ Thus, $m(T|_Y)^2=\|Ty\|^2=\mu^2(\alpha^2+\beta^2)+\gamma^2=\mu^2\|T\|^2+\gamma^2.$ So $\mu^2=\frac{m(T|_Y)^2-\gamma^2}{\|T\|^2}\leq 1.$ Since $x\in V_T,$ from Theorem \ref{Th-vt}, we have, $x$ is an eigenvector of $Re(T).$ Hence, $\langle Re(T)x,y\rangle=0\Rightarrow \langle Tx,y\rangle+\langle Ty,x\rangle=0\Rightarrow \beta+\mu\beta=0\Rightarrow1+\mu=0,$ since $\beta\neq 0.$ Therefore, $|\mu|=1\Rightarrow \frac{m(T|_Y)^2-\gamma^2}{\|T\|^2}=1\Rightarrow \|T\|^2=m(T|_Y)^2-\gamma^2\leq m(T|_Y)^2\leq \|T\|^2.$ Thus, $m(T|_Y)=\|T\|=\|T|_Y\|$ and $\gamma=0.$ This gives that $T:Y\to Y$ is a scalar multiple of isometry, $\|T\|=\|T|_Y\|$ and $v(T)=v(T|_Y).$ We now show that $Y\subseteq (\ker(T))^{\perp}.$ Let $z\in S_Y.$ Then there exists $z_1\in \ker(T)$ and $z_2\in (\ker(T))^{\perp}$ such that $z=z_1+z_2.$ So $\|z_1\|^2+\|z_2\|^2=1.$ Now, $\|T\|=\|T|_Y\|=\|Tz\|=\|Tz_1+Tz_2\|=\|Tz_2\|\Rightarrow \|z_2\|=1$ and $z_1=0.$ Thus, $z\in (\ker(T))^{\perp}.$ This completes the proof of the necessary part of the theorem.\\
	Now we prove the sufficient part of the theorem. $(i)$ clearly implies that $x\in M_T\cap V_T.$ If $(ii)$ holds, then $M_{T|_Y}=S_Y\Rightarrow M_{T|_Y}\cap V_{T|_Y}\neq \emptyset.$ Let $z\in M_{T|_Y}\cap V_{T|_Y}.$ Then $\|T\|=\|T|_Y\|=\|Tz\|\Rightarrow z\in M_T$ and $v(T)=v(T|_Y)=|\langle Tz,z\rangle|\Rightarrow z\in V_T.$ Thus, $M_T\cap V_T\neq \emptyset.$ This completes the proof of the theorem.	
\end{proof}

In case of rank $ 1 $ bounded linear operators on a real Hilbert space, the following two corollaries are now evident.

\begin{cor}\label{cor:rank1}
Let $\mathbb{H}$ be a real Hilbert space. Let $T\in L(\mathbb{H})$ be of rank $1$. Then $M_T \cap V_T \neq \emptyset$ if and only if there exists $x\in S_{\mathbb{H}}$  such that $Tx=\pm \|T\| x.$
\end{cor}

\begin{cor}
Let $\mathbb{H}$ be a real Hilbert space. Let $T\in L(\mathbb{H})$ be of rank $1$. If $M_T \cap V_T \neq \emptyset$ then $v(T)=\|T\|$.
\end{cor}

For rank $ 2 $ operators on a real Hilbert space, we have the following corollary:

\begin{cor}\label{cor:rank2}
Let $\mathbb{H}$ be a real Hilbert space. Let $T\in S_{L(\mathbb{H}))}$ be of rank $2$. Then $M_T \cap V_T \neq \emptyset$ if and only if either of the following is true:\\
$(i)$ there exists $x\in S_{\mathbb{H}}$  such that $Tx=\pm x$.\\
$(ii)$ $T$ is a partial isometry whose initial subspace and final subspace are identical.
\end{cor}
\begin{proof}
Observe that for a rank $2$ operator $T$ condition $(ii)$ of Theorem \ref{th-mtvt} is equivalent to $T$ is a partial isometry whose initial subspace and final subspace are identical.
\end{proof}
 
Our next result illustrates that the situation can be quite different in case of bounded linear operators on a normed space. This further illustrates the speciality of the geometric structure of Hilbert spaces, in contrast to general normed spaces. We recall that a normed space $ \mathbb{X} $ is said to be polygonal if the unit ball $ B_{\mathbb{X}} $ contains only finitely many extreme points.

\begin{theorem}
Let $\mathbb{X}$ be a finite-dimensional polygonal Banach space. Then there always exists an operator $T\in L(\mathbb{X})$ such that rank$(T)=1,~M_T\cap V_T\neq \emptyset$ but there does not exist any $x\in S_{\mathbb{X}}$ such that $Tx=\pm \|T\|x.$ 
\end{theorem}
\begin{proof}
	We first prove the theorem for  $2-$dimensional polygonal Banach spaces. There exist three vertices $v_1,v_2,v_3$ in $S_{\mathbb{X}}$ such that $\|(1-t)v_1+tv_2\|=\|(1-t)v_2+tv_3\|=1$ for all $t\in [0,1].$ It is easy to observe that there exists a non-zero vector $h$ in the line segment joining $v_3$ and $\frac{v_3-v_2}{\|v_3-v_2\|}$ such that $v_2\perp_B h$ and $\|v_2+\lambda h\|>1$ for all $\lambda >0.$  Define a linear operator $T$ on $\mathbb{X}$ by $T(v_2)=\frac{v_2+v_3}{2}$ and $T(h)=0.$ Then clearly, $\|T\|=1=v(T)$ and $v_2\in M_T\cap V_T.$  Clearly, $\frac{v_2+v_3}{2}=\frac{tv_2+(1-t)h}{\|tv_2+(1-t)h\|}$ for some $t\in (0,1).$ Now, let $x\in S_{\mathbb{X}}.$ Then $x=av_2+bh$ for some $a,b\in \mathbb{R}.$ If possible, suppose that $Tx=x.$ Then $aTv_2=av_2+bh\Rightarrow a\frac{v_2+v_3}{2}=av_2+bh\Rightarrow a\frac{tv_2+(1-t)h}{\|tv_2+(1-t)h\|}=av_2+bh.$ From this equality it follows that $\|tv_2+(1-t)h\|=t,$ which contradicts that $\|v_2+\lambda h\|>1$ for all $\lambda>0.$ Similarly, $Tx=-x$ gives that   $\|tv_2+(1-t)h\|=-t,$ which is clearly a contradiction. Thus, there exists no $x\in S_{\mathbb{X}}$ such that $Tx=\pm x.$ This completes the proof for  $2-$dimensional polygonal Banach space.\\
	Now, suppose that dimension of $\mathbb{X}=n(>2).$ Let $v$ be an extreme point of $\mathbb{X}.$ Let $Y$ be any $2-$dimensional subspace of $\mathbb{X}$ containing $v.$ Then $Y$ is a polygonal $2-$dimensional Banach space and $v$ is an extreme point of $Y.$ There exist two vertices $v_1,v_3\in S_{Y}$ such that $\|(1-t)v_1+tv\|=\|(1-t)v+tv_3\|=1$ for all $t\in [0,1].$ Now, as in $2-$dimensional case we choose $h\in S_{Y}$ such that $v\perp_B h$ and $\|v+\lambda h\|>1$ for all $\lambda >0.$ There exists a hyperspace $H$ in $\mathbb{X}$ such that $x\perp_B H$ and $h\in H.$ Define $T\in L(\mathbb{X})$ by $T(v)=\frac{v+v_3}{2}$ and $T(h)=0$ for all $h\in H.$ Then as previous, it can be shown that $M_T\cap V_T\neq \emptyset$  but there does not exist any $x\in S_{\mathbb{X}}$ such that $Tx=\pm x.$
\end{proof}

Our next result extends Corollary \ref{cor:rank1} to a large class of strictly convex Banach spaces.

\begin{theorem}
	Let $\mathbb{X}=\ell_p(\mathbb{R}^2),$ where $p$ in even. Let $T\in S_{L(\mathbb{X})}$ be such that rank$(T)=1$ and $M_T\cap V_T\neq \emptyset.$ Then there exists $(x,y)\in S_{\mathbb{X}}$ such that $T(x,y)=\pm (x,y).$
\end{theorem}
\begin{proof}
	First suppose that $e=(1,0)\in M_T\cap V_T.$ Since $(1,0)\perp_B(0,1),$ $T(1,0)\perp_BT(0,1).$ Now, rank$(T)=1$ gives that $T(0,1)=(0,0).$ Let $Te=(\alpha,\beta).$ Then $\|Te\|=1\Rightarrow \alpha^p+\beta^p=1.$ Let $e^*\in S_{\mathbb{X}^*}$ be such that $e^*(e)=1.$ Then $|e^*(Te)|=|e^*(\alpha,\beta)|=|\alpha|.$ If $|\alpha|=1,$ then $T(1,0)=\pm(1,0)$ and we are done. If possible, suppose that $|\alpha|<1.$ Then $\beta\neq0.$ Choose $w=(\frac{p\alpha}{k},\frac{(p-1)\beta}{k}),$ where $k=\{p^p\alpha^p+(p-1)^p\beta^p\}^{\frac{1}{p}}.$ Then $w\in S_{\mathbb{X}}.$ Now, $Tw=\frac{p\alpha}{k}T(1,0)=\frac{p\alpha}{k}(\alpha,\beta).$ By an easy calculation it can be observed that if $w^*\in S_{\mathbb{X}^*}$ be such that $w^*(w)=1$ then $|w^*(Tw)|=|\alpha||\frac{p^p\alpha^p+\frac{p}{p-1}(p-1)^p\beta^p}{k^p}|>|\alpha|=|e^*(Te)|,$ since $\frac{p}{p-1}>1$ and $\beta\neq0.$ This contradicts that $e\in V_T.$ Thus, $|\alpha|=1$ and $T(1,0)=\pm(1,0).$ Similarly, if $(0,1)\in M_T\cap V_T,$ then it can be shown that $T(0,1)=\pm (0,1).$\\
	Now, suppose that $e=(x,y)\in M_T\cap V_T,$ where $x\neq0,y\neq0.$ Then $(x,y)\perp_B(y^{p-1},-x^{p-1}).$ Hence $T(x,y)\perp_BT(y^{p-1},-x^{p-1}).$ Since rank$(T)=1,$ we have, $T(y^{p-1},-x^{p-1})=0.$ Let $Te=T(x,y)=(\alpha,\beta).$ Then $\alpha^p+\beta^p=1.$ Let $w=(\gamma,\delta)\in S_{\mathbb{X}}.$ Then $\gamma^p+\delta^p=1.$ Now, $w=(\gamma,\delta)=a(x,y)+b(y^{p-1},-x^{p-1}),$ implies that $a=\gamma x^{p-1}+\delta y^{p-1}.$ Thus, $Tw=a(\alpha,\beta).$ Now, $w^*\in S_{\mathbb{X}^*}$ such that $w^*(w)=1$ implies that $|w^*(Tw)|=|\alpha \gamma^px^{p-1}+\alpha\delta y^{p-1}\gamma^{p-1}+\beta\gamma x^{p-1}\delta^{p-1}+\beta y^{p-1}\delta^p|.$ Now, we maximize $|w^*(Tw)|$ subject to the condition $\gamma^p+\delta^p=1.$ Suppose $F(\gamma,\delta)=\alpha \gamma^px^{p-1}+\alpha\delta y^{p-1}\gamma^{p-1}+\beta\gamma x^{p-1}\delta^{p-1}+\beta y^{p-1}\delta^p-\lambda(\gamma^p+\delta^p-1),$ where $\lambda$ is Lagrange multiplier. Now, for extremum of $w^*(Tw),$ we have $F_\gamma(\gamma,\delta)=0$ and $F_\delta(\gamma,\delta)=0.$ Since $(x,y)\in V_T,$ we get $F_\gamma(x,y)=0$ and $F_\delta(x,y)=0.$ Now, $F_\gamma(x,y)=0\Rightarrow$
	\begin{eqnarray}
	\lambda px^p=\alpha p x^{2p-1}+\alpha(p-1)y^px^{p-1}+\beta x^py^{p-1}.
	\end{eqnarray}
	$F_\delta(x,y)=0\Rightarrow$
	\begin{eqnarray}
	\lambda py^p=\alpha  x^{p-1}y^p+\beta(p-1)x^py^{p-1}+\beta py^{2p-1}.
	\end{eqnarray}
	Solving equation $(1)$ and $(2),$ we get $(\alpha,\beta)=\pm(x,y).$ Thus, $T(x,y)=\pm(x,y).$ This completes the proof of the theorem. 
\end{proof}

We next consider a bounded linear operator $ T $ on a strictly convex normed space that satisfies the Daugavet equation. In this context, we have the following result: 

\begin{theorem}\label{theorem:scspace}
Let $\mathbb{X}$ be a strictly convex normed space. Suppose $T\in L(\mathbb{X})$ is such that $(i)$ $I+T$ attains its norm and $(ii)$ $T$ satisfies the Daugavet equation. Then the following holds:\\ $(a)$ $T$ attains its norm,\\ $(b)$ $\|T\|$ is an eigenvalue of $T$,\\ $(c)$ $\|T\|=v(T)$,\\  $(d)$ $M_T \cap V_T\neq \emptyset.$
\end{theorem}

\begin{proof}
From the given condition (i), we have  there exists $x\in S_{\mathbb{X}}$ such that $\|(I+T)x\|=\|I+T\|.$ Then, from the given condition (ii) we have $1+\|T\|=\|I+T\|=\|(I+T)x\|=\|x+Tx\|\leq \|x\|+\|Tx\|=1+\|Tx\|\leq 1+\|T\|.$ From this inequality it follows that $\|Tx\|=\|T\|$ and $\|x+Tx\|=\|x\|+\|Tx\|$. Therefore, from the equality condition of strictly convex space we have $Tx=\lambda x$, for some $\lambda \geq 0$. So, $\|Tx\|=\lambda=\|T\|$. Therefore, $Tx=\|T\|x.$ Now, from Hahn-Banach Theorem we have, there exists $x^{*}\in S_{\mathbb{X}^{*}}$ with $x^{*}(x)=1$ such that $|x^{*}(Tx)|=\|T\|$. This imply that $v(T)=\|T\|.$ Also, we see that $x\in M_T \cap V_T $. This completes the proof of the theorem.
\end{proof}

\begin{remark}
	Theorem \ref{theorem:scspace} may not be true in a normed space which is not strictly convex. For example, if we consider $T\in L(\ell_{\infty}(\mathbb{R}^2))$ defined by $T(x,y)=(\frac{x+y}{2},0),$ then it is easy to observe that $T$ satisfies the Daugavet equation but $\|T\|$ is not an eigenvalue of $T.$  
\end{remark}

We next present an easy sufficient condition for the intersection of the norm attainment set and the numerical radius attainment set of a bounded linear operator to be non-empty.

\begin{theorem}\label{theorem:nonempty}
Let $\mathbb{X}$ be a normed space. Let $T\in S_{L(\mathbb{X})}$ with $V_T \neq \emptyset$. If $\|T\|=v(T),$ then  $M_T \cap V_T\neq \emptyset.$
\end{theorem}
\begin{proof}
Let $\|T\|=v(T)$. Since $V_T \neq \emptyset$, so let $x\in V_T$. Then there exists $x^{*}\in S_{\mathbb{X}^{*}}$ such that $x^{*}(x)=1$ and $|x^{*}(Tx)|=v(T)$. Now, $\|T\|=v(T)=|x^{*}(Tx)|\leq \|Tx\|\leq \|T\|.$ This shows that  $\|Tx\|=\|T\|$. So, $x\in M_T.$  Therefore, $V_T \subseteq M_T$. Since $V_T \neq \emptyset$, so $M_T \cap V_T\neq \emptyset$.
\end{proof}

\begin{remark}
We note that, the converse part of the Theorem \ref{theorem:nonempty} may not true. For example, if we consider $T\in L(\ell_2(\mathbb{R}^4))$ defined by \[T(x,y,z,w)=\Big(\frac{x-y-z}{\sqrt{3}},\frac{x+y}{\sqrt{3}},\frac{x+z}{\sqrt{3}},0\Big),\]
then it is easy to see that $(1,0,0,0)\in M_T\cap V_T,~\|T\|=1$ and $v(T)=\frac{1}{\sqrt{3}},$ i.e., $M_T\cap V_T\neq \emptyset$ but $\|T\|\neq v(T).$
\end{remark}

In the following theorem we completely characterize when the norm of a bounded linear operator is equal to its numerical radius, in the context of strictly convex normed spaces.

\begin{theorem}\label{theorem:com nonempty}
Let $\mathbb{X}$ be a strictly convex normed space. Let $T\in S_{L(\mathbb{X})}$ with $V_T \neq \emptyset$. Then $\|T\|=v(T)$ if and only if there exists $x \in S_{\mathbb{X}}$ such that $Tx=\lambda x$ for some scalar $\lambda$ with $|\lambda|=1$.
\end{theorem}

\begin{proof}
First we prove the sufficient part of the theorem. Suppose there exists $x \in S_{\mathbb{X}}$ such that $Tx=\lambda x$ for some scalar $\lambda$ with $|\lambda|=1$. Then $\|Tx\|=|\lambda|=1=\|T\|.$ Also, by Hahn-Banach theorem we have, there exists $x^{*}\in S_{\mathbb{X}^{*}}$ with $x^{*}(x)=1$ such that $|x^{*}(Tx)|=|\lambda|=1$. So, $v(T)=1$. Therefore, $\|T\|=v(T).$ This completes the proof of the sufficient part of the theorem.\\
Now we prove the necessary part of the theorem. Let $\|T\|=v(T)$. Suppose $x\in V_T$. So, there exists $x^{*}\in S_{\mathbb{X}^{*}}$ such that $x^{*}(x)=1$ and $|x^{*}(Tx)|=v(T)=1$. This implies that $x^{*}(Tx)=e^{i\theta}$, for some $\theta \in [0,2\pi)$.  Now, $\|T\|=v(T)=|x^{*}(Tx)|\leq \|Tx\|\leq \|T\|.$ This shows that  $\|Tx\|=\|T\|,$ i.e., $x\in M_T.$ If possible, let $\{x, Tx\}$ be linearly independent. Also, let $k=\|tx+e^{-i\theta}(1-t)Tx\|$. Since $\mathbb{X}$ is strictly convex space, so $k<1.$ Now, $x^{*}(\frac{tx+e^{-i\theta}(1-t)Tx}{k})=\frac{1}{k}>1$ as  $k<1$. This is a contradiction because  $\frac{tx+e^{-i\theta}(1-t)Tx}{k} \in S_{\mathbb{X}}.$ So, $\{x, Tx\}$ is linearly dependent. Therefore, there exists a scalar $\lambda$ such that $Tx=\lambda x.$ Since $x\in M_T$ and $\|T\|=1$, so $|\lambda|=1.$ This completes the proof of the necessary part of the theorem. 
\end{proof}

\begin{remark} 
If $\mathbb{X}$ is not strictly convex then the necessary part of the Theorem \ref{theorem:com nonempty} may not be true. This is illustrated with the following examples.\\

\noindent \textbf{Example 1.}
Let $T: \ell_1(\mathbb{R}^n)\rightarrow \ell_1(\mathbb{R}^n)$ be a bounded linear operator defined as 
\begin{eqnarray*}
	T(x_1,x_2,\ldots,x_n)&=&\Big(\frac{x_1}{n},\frac{x_1}{n},\ldots,\frac{x_1}{n}\Big).
\end{eqnarray*}
Then, it is easy to check that $\|T\|=v(T)=1$ and $(1,0,\ldots,0)\in M_T \cap V_T$. But there exists no $x\in S_{\ell_1(\mathbb{R}^n)}$ such that $Tx=\pm x.$\\

\noindent \textbf{Example 2.}
Let $T: \ell_\infty(\mathbb{R}^n)\rightarrow \ell_\infty(\mathbb{R}^n)$ be a bounded linear operator defined as 
\begin{eqnarray*}
		T(x_1,x_2,\ldots,x_n)&=&\Big(\frac{x_1+x_2+\ldots+x_n}{n},0,\ldots,0\Big).
\end{eqnarray*}
Then, it is easy to check that $\|T\|=v(T)=1$ and $(1,1,\ldots,1)\in M_T \cap V_T$. But there exists no $x\in S_{\ell_\infty(\mathbb{R}^n)}$ such that $Tx=\pm x.$
\end{remark}

In the next three theorems, we discuss the intersection properties of the norm attainment set (minimum norm attainment set) and the Crawford number attainment set of a bounded linear operator on a real Hilbert space.

\begin{theorem}
	Let $\mathbb{H}$ be a real Hilbert space and $T\in L(\mathbb{H}).$ Then $M_T\cap c_T\neq \emptyset$ if and only if there exists $x\in S_{\mathbb{H}}$ such that for any $z\in x^{\perp}\cap S_{\mathbb{H}},$ the following conditions hold:\\
	(i)  $\langle Tx,Tz\rangle=0,$ \\
	(ii) $\|Tz\|\leq \|Tx\|.$  \\
  (iii) either $\langle Tx,x\rangle =0$ or if $\langle Tx,x\rangle\neq 0,$ then $x$ is an eigenvector of $Re(T)$ and $|\langle Tx,x\rangle| \leq |\langle Tz,z\rangle|.$ 
\end{theorem}	
\begin{proof}
	The result follows from Theorem \ref{Th-ct} and \cite[Th. 2.1]{S}. 
\end{proof}

\begin{theorem}
	Let $\mathbb{H}$ be a real Hilbert space and $T\in L(\mathbb{H})$ be such that $c_T\neq \emptyset.$ Then $M_T\cap c_T\neq\emptyset$ if and only if either of the following conditions holds:\\
	(i) $c(T)=v(T).$\\
	(ii) there exists $u,v\in S_{\mathbb{H}}$ such that $u\perp v$ and $Tu=\pm \|T\|v.$ \\
	(iii) there exists a two dimensional subspace $Y$ of $(ker(T))^{\perp}$ such that $T:Y\to Y$ is a scalar multiple of isometry, $\|T\|=\|T|_{Y}\|$ and $c(T)=c(T|_Y).$	
\end{theorem}
\begin{proof}
	The proof follows using similar arguments as given in the proof of Theorem \ref{th-mtvt}.
\end{proof}

\begin{theorem}
	Let $\mathbb{H}$ be a real Hilbert space and $T\in L(\mathbb{H}).$ Then $m_T\cap c_T\neq \emptyset$ if and only if there exists $x\in S_{\mathbb{H}}$ such that for any $z\in x^{\perp}\cap S_{\mathbb{H}},$  the following conditions hold:\\
	(i)  $\langle Tx,Tz\rangle=0,$ \\
	(ii) $\|Tz\|\geq \|Tx\|.$ \\
	(iii) either $\langle Tx,x\rangle =0$ or if $\langle Tx,x\rangle\neq 0,$ then $x$ is an eigenvector of $Re(T)$ and $|\langle Tx,x\rangle| \leq |\langle Tz,z\rangle|.$
\end{theorem}
	
	\begin{proof}
		The result follows from Theorem \ref{Th-ct} and \cite[Th. 2.4]{SPM}. 
	\end{proof}

In case of a real $ 2- $dimensional Hilbert space, we have the following characterization:

\begin{theorem}
	Let $\mathbb{H}$ be a $2-$dimensional real Hilbert space and $T\in L(\mathbb{H}).$ Then $m_T\cap c_T\neq\emptyset$ if and only if either of the following conditions holds:\\
	(i) there exists $x\in S_{\mathbb{H}}$ such that $Tx=\pm m(T)x=\pm c(T)x.$\\
	(ii) there exists $u,v\in S_{\mathbb{H}}$ such that $u\perp v$ and $Tu=\pm m(T)v.$\\
	(iii) $T$ is a scalar multiple of isometry.	
\end{theorem}
\begin{proof}
	The sufficient part follows trivially. We only prove the necessary part of the theorem. Let $m_T\cap c_T\neq\emptyset.$ Let $x\in m_T\cap c_T$ and $y\in M_T.$ Then  by \cite[Th. 3.1, Th. 2.4]{SPM}, we have, $x\perp y$ and $Tx\perp Ty.$  Suppose that $Tx=\alpha x+\beta y,$ for some $\alpha,\beta\in \mathbb{R}.$ If $\beta=0,$ then $x\in m_T\Rightarrow m(T)=\|Tx\|=|\alpha|$ and $x\in c_T\Rightarrow c(T)=|\langle Tx,x\rangle|=|\alpha|.$ Thus, in this case $(i)$ holds. If $\alpha=0,$ then clearly $(ii)$ holds. Now, let $\alpha \neq 0,\beta \neq 0.$ Since, $Tx\perp Ty,$  $Ty=\mu(\beta x-\alpha y),$ for some $\mu\in \mathbb{R}.$ Thus, $\|T\|^2=\|Ty\|^2=\mu^2(\alpha^2+\beta^2)=\mu^2m(T)^2.$ So $\mu^2=\frac{\|T\|^2}{m(T)^2}\geq 1.$ Since $x\in c_T$ and $\langle Tx,x\rangle=\alpha\neq 0,$ by Theorem \ref{Th-ct}, we get, $x$ is an eigenvector of $Re(T).$ Hence, $\langle Re(T)x,y\rangle=0\Rightarrow \langle Tx,y\rangle+\langle Ty,x\rangle=0\Rightarrow \beta+\mu\beta=0\Rightarrow1+\mu=0,$ since $\beta\neq 0.$ Therefore, $|\mu|=1\Rightarrow \|T\|=m(T).$ This gives that $T$ is a scalar multiple of isometry. Thus, in this case $(iii)$ holds. This completes the proof of the necessary part of the theorem.\\
	Now we prove the sufficient part of the theorem. $(i)$ clearly implies that $x\in m_T\cap c_T.$ If $(ii)$ holds, then $u\in m_T$ and $|\langle Tu,u\rangle|=0\Rightarrow c(T)=0$ and $u\in c_T.$ So $m_T\cap c_T\neq \emptyset.$ If $(iii)$ holds, then $m_T=S_{\mathbb{H}}.$ Since $c_T\neq \emptyset, m_T\cap c_T\neq \emptyset.$ This completes the proof of the theorem.
\end{proof}

We end the present article with an easy sufficient condition for the non-empty intersection of the minimum norm attainment set and the Crawford number attainment set of a bounded linear operator on a normed space.

\begin{theorem}
	Let $\mathbb{X}$ be a normed space. Let $T\in S_{L(\mathbb{X})}$ be such that $m_T \neq \emptyset$. If $c(T)=m(T),$ then  $m_T \cap c_T\neq \emptyset.$
\end{theorem}
\begin{proof}
	Let $x\in m_T$ and $x^*\in S_{\mathbb{X}^*}$ be such that $x^*(x)=1.$ Then $c(T)=m(T)=\|Tx\|\geq |x^*(Tx)|\geq c(T)$ implies that $|x^*(Tx)|= c(T).$ Thus, $x\in c_T.$ Hence, $m_T\cap c_T\neq\emptyset.$
\end{proof}

\bibliographystyle{amsplain}

\end{document}